\documentclass[a4paper,12pt,oneside,bibliography=totoc]{article}

\usepackage{ucs}                     
\usepackage[utf8x]{inputenc}         
\usepackage[T1]{fontenc}             
\usepackage[english]{babel}          
\usepackage{amsmath, amssymb, amstext, bm}    
\usepackage[automark]{scrpage2}
\usepackage{lmodern}
\usepackage{textcomp}                
\usepackage{caption}                 
\usepackage{subcaption}              
\usepackage{multirow}
\usepackage{enumitem}

\usepackage{t1enc}
\usepackage{verbatim} 
\usepackage{graphicx}
\usepackage[dvips]{epsfig}
\usepackage{url}
\usepackage{titlesec}
\usepackage{pdfpages}
\usepackage[nottoc,notlot,notlof]{tocbibind}
\usepackage{hyperref}
\usepackage{url}

\titleformat{\chapter}
{\normalfont\Large\bfseries}{\thechapter}{1em}{}
\titlespacing*{\chapter}{0pt}{3.5ex plus 1ex minus .2ex}{2.3ex plus .2ex}

\titleformat{\section}
{\normalfont\large\bfseries}{\thesection}{1em}{}

\titleformat{\subsection}
{\normalfont\normalsize\bfseries}{\thesubsection}{1em}{}

\usepackage[thmmarks,amsmath,hyperref,noconfig]{ntheorem} 
\newtheorem{thm}{Theorem} 
\newtheorem{lemma}{Lemma}

\newtheorem{prop}{Proposition}

\theorembodyfont{\upshape}

\theoremstyle{nonumberplain}
\theoremheaderfont{\itshape}
\theorembodyfont{\normalfont}
\theoremseparator{.}
\theoremsymbol{\ensuremath{_\square}}
\newtheorem{proof}{Proof}
\qedsymbol{\ensuremath{_\square}}

\theoremstyle{empty}
\theoremsymbol{\ensuremath{_\square}}
\newtheorem{refproof}{Proof}
\qedsymbol{\ensuremath{_\square}}
\newcommand{\R}{\mathbb{R}}
\newcommand{\N}{\mathbb{N}}
\newcommand{\Z}{\mathbb{Z}}
\newcommand{\Q}{\mathbb{Q}}

\newcommand{\lcm}{\rm lcm}
\newcommand{\F}{\mathbb{F}}

\newcommand{\NN}{\mathbb{N}}

\newcommand{\FF}{\mathbb{F}}

\newcommand{\ui}{[0,1)}

\newcommand{\xn}{(\bm{x}_n)_{n \geq 0}}
\newcommand{\lc}{_{n\geq 0}}

\newcommand{\bse}{\bm{e}}
\newcommand{\bsx}{\bm{x}}
\newcommand{\vol}{{\rm vol}}
\newcommand{\bsz}{\boldsymbol{z}}

\begin{document}
	
\title{Pair Correlations of Niederreiter and Halton Sequences are not Poissonian}
\author{Roswitha Hofer\footnote{The author is supported by the Austrian Science Fund (FWF), Project F5505-N26, which is a part of the Special Research Program Quasi-Monte Carlo Methods: Theory and Applications.}, Lisa Kaltenböck\footnote{The author is supported by the Austrian Science Fund (FWF), Project F5507-N26, which is a part of the Special Research Program Quasi-Monte Carlo Methods: Theory and Applications.}}
\date{\vspace{-5ex}}
\maketitle

\begin{abstract}
Niederreiter and Halton sequences are two prominent classes of multi-dimensional sequences which are widely used in practice for numerical integration methods because of their excellent distribution qualities. In this paper we show that these sequences---even though they are uniformly distributed---fail to satisfy the stronger property of Poissonian pair correlations. This extends already established results for one-dimensional sequences and confirms a conjecture of Larcher and Stockinger. The proofs rely on a general tool which identifies specific regularities of a sequence to be sufficient for not having Poissonian pair correlations. 
\end{abstract}

\section{Introduction}
Let $\| \cdot \|$ denote the distance to the nearest integer. A sequence $(x_n)\lc$ of real numbers in $\ui$ has Poissonian pair correlations if
$$\frac{1}{N} \#\left\{0 \leq n \neq l \leq N-1 : \|x_n - x_l\| \leq \frac{s}{N}\right\} \rightarrow 2s$$
for every real number $s \geq 0$ as $N \to \infty$.

The investigation of pair correlations of sequences was originally motivated by problems in quantum chaos, see e.g. \cite{AAL18} and the references therein. In the last few years, in particular the case of Poissonian pair correlations has also been studied from a purely mathematical point of view as this property is natural for a sequence of independently chosen random numbers drawn from the uniform distribution. Extensive research recently has been done in terms of metrical theory as well as for concrete sequences. An introduction to this topic and a collection of results is provided by \cite{LS2019}.

For example, it is known that any sequence $(x_n)\lc$ in $\ui$ which has Poissonian pair correlations is also uniformly distributed, i.e.
$$\lim_{N \to \infty} \frac{1}{N} \#\{0 \leq n \leq N -1: x_n \in [a,b)\} = b-a$$
for all $0 \leq a < b \leq 1$ (this was independently proven in \cite{ALP18,GreLar,Stei17}). However, the converse is not true since for many explicit examples of classical low-discrepancy sequences, such as the Kronecker sequence $(\{n\alpha\})\lc$, the van der Corput sequence and certain digital $(t,1)$-sequences, it has been shown that they do not have Poissonian pair correlations (see e.g. \cite{LS2018}).

A generalization of most previously mentioned results to a multi-dimensional setting has recently been established in \cite{HinrEtAl2018}, see also \cite{Stei19} and \cite{Marklof19} for a slightly different approach and a more general analysis of higher dimensional pair correlations. In this work, we refer to the former concept. Therefore, let $\| \cdot \|_\infty$ denote a combination of the supremum-norm of a $d$-dimensional vector $\bm{x} = (x^{(1)}, \dots, x^{(d)}) \in \R^d$ with the distance to the nearest integer function $\|\cdot\|$ defined by 
$$\| \bm{x} \|_\infty := \max(\|x^{(1)}\|, \dots, \|x^{(d)}\|).$$ 
A $d$-dimensional sequence $\xn \in \ui^d$ has Poissonian pair correlations if
\begin{equation}\label{equ:def_d}
\frac{1}{N} \#\left\{0 \leq n \neq l \leq N-1 : \|\bm{x}_n - \bm{x}_l\|_{\infty} \leq \frac{s}{N^{1/d}}\right\} \rightarrow (2s)^d 
\end{equation}
for every real number $s \geq 0$ as $N \to \infty$.

In analogy to the one-dimensional case it could be shown that sequences with this property are uniformly distributed in $\ui^d$ and that the $d$-dimensional Kronecker sequence $(\{n\bm{\alpha}\})\lc$ does not have Poissonian pair correlations for any $\bm{\alpha} \in \R^d$. However, whether the multi-dimensional analogues of other well-distributed one-dimensional point sequences---as Halton or digital $(t,d)$-sequences---have Poissonian pair correlations remained unanswered.

The motivation behind the present paper was to close this gap. Before the presentation of our results, let us briefly review the construction of the considered sequences.

\subsection{Digital $(t,d)$-sequences}
\label{sec:digital_sequences}
A widely used class of low-discrepancy sequences for numerical integration methods are $(t,d)$-sequences that are generated via the digital method:
\\

\textit{
	Let $d \in \NN$, let $\F_q$ be a finite field with $q$ elements and characteristic $p$ and let $\phi: \F_q \to \{0, 1, \dots, q-1\}$ be a bijection satisfying $\phi(0)=0$. Further, let $C^{(1)},\ldots, C^{(d)}\in\FF_q^{\NN\times \NN_0}$, be given generating matrices, where we assume that each column of such a matrix contains only finitely many nonzero entries. We construct a sequence $\xn$, $\bm{x}_n = (x_n^{(1)}, \dots, x_n^{(d)})$, by generating the $j$-th component of the $n$-th point, $x_n^{(j)}$, as follows. We represent $n = n_0 + n_1q + n_2q^2 + \dots$ in base $q$ and set
	$$C^{(j)} \cdot (\phi^{-1}(n_0), \phi^{-1}(n_1), \dots)^\top =: (y^{(j)}_1, y^{(j)}_2, \dots)^\top \in \FF_q^{\NN}$$
	and 
	$$x_n^{(j)} := \sum_{i \geq 1} \frac{\phi(y_i^{(j)})}{q^{i}}.$$
}

The distribution properties of the constructed sequence $\xn$ strongly depend on the generating matrices $C^{(1)},\ldots, C^{(d)}$.

If for every $m \in \NN$ and for all $r_1, \dots, r_d \in \NN_0$ with $r_1 + \dots + r_d = m-t$, $t \in \NN_0$, 
the set of row vectors
$$\{(c^{(j)}_{r,k})_{0 \leq k < m}: j \in \{1, \dots, d\}, r \in \{1, \dots, r_j\}\}$$
is linearly independent over $\F_q$, then $\xn$ is a $(t,d)$-sequence in base $q$, i.e. for all integers $m > t$ and $s \in \NN_0$ the point set $(\bm{x}_n)_{sq^m \leq n < (s+1)q^m}$ has the property that any elementary interval of order $m-t$, that is any interval of the form
$$I(v_1, \dots, v_d) := \prod_{j = 1}^d \Big[\frac{a_j}{q^{v_j}},\frac{a_j+1}{q^{v_j}}\Big)$$
with $v_1 + \dots + v_d = m-t$ and $a_j \in \{0,1, \dots, q^{v_j}-1\}$, 
contains exactly $q^t$ points. For more detailed information on $(t,d)$-sequences and their construction we refer to \cite{DP10,N92} and the references therein.

The more specific notion of $(t, \bse, d)$-sequences, where $t \in \NN_0$ and $\bse = (e_1, \dots, e_d)$, was introduced by Tezuka in \cite{Tez13} and analyzed e.g. in \cite{HoNi13}. For constructing such sequences, for all $r_1, \dots, r_d \in \NN_0$ such that $e_1r_1 + \dots + e_dr_d \leq m-t$, the set of row vectors
$$\{(c^{(j)}_{r,k})_{0 \leq k < t + \sum_{j=1}^{d}e_jr_j}: j \in \{1, \dots, d\}, r \in \{1, \dots, e_jr_j\}\}$$
has to be linearly independent over $\FF_q$.
Then, for all integers $m > t$ and $s \in \NN_0$ the point set $([\bm{x}_n]_{q,m})_{sq^m \leq n < (s+1)q^m}$, where $[\bm{x}_n]_{q,m}$ denotes the coordinate-wise $q$-ary $m$-digit truncation of the point $\bm{x}_n$, satisfies that for any $v_1, \dots, v_d \in \NN_0$ with $e_1v_1 + \dots + e_d v_d \leq m-t$, the elementary interval $I(e_1v_1, \dots, e_dv_d)$ contains exactly $q^{m - (e_1v_1 + \dots + e_d v_d)}$ points.

Note that a $(t,d)$-sequence in base $q$ is identical to a $(t,\bse, d)$-sequence in base $q$ if $\bse = (1, \dots, 1)$. Moreover, it is known that any $(u,\bse,d)$-sequence in base $q$ is a $(t,d)$-sequence in base $q$ with $t = u + \sum_{j = 1}^d (e_j-1)$.

It is a non-trivial task to find or construct matrices satisfying such strict conditions on their rank structure. One famous example was given by Niederreiter \cite{Nie88}.
\\

\textit{	
For a given dimension $d \in \NN$ we choose $q_1(x), \dots, q_d(x) \in \FF_q[x]$ to be monic non-constant pairwise co-prime polynomials over $\FF_q$ of degrees $e_j:=\deg{q_j(x)} \geq 1$ for $j \in \{1, \dots, d\}$. Set $\bse=(e_1, \dots, e_d)$. Now the $i$-th row of the $j$-th generating matrix $C^{(j)}$, denoted by $\rho^{(j)}_i$, is constructed as follows. We choose $s \in \NN$ and $r \in \{0, \dots, e_j-1\}$ such that $i = e_j s - r$, consider the expansion 
$$\frac{x^r}{q_j(x)^s}=\sum_{k \geq 0} a^{(j)}(s,r,k)x^{-k-1} \in \FF_q((x^{-1}))$$ and set 
$\rho^{(j)}_i=(a^{(j)}(s,r,k))_{k\geq 0}$. 
}

It is easy to check that the generating matrices are non-singular upper triangular $(NUT)$ matrices over $\FF_q$. Furthermore, they generate digital $(0,\bse,d)$- and $(t,d)$-sequences over $\FF_q$ with $t = \sum_{j = 1}^{d}(e_j-1)$ (see e.g. \cite{Tez13, HoNi13}).

An alternative column by column construction method for generating matrices was introduced in \cite{hofer}.
\\

\textit{
We choose $d \in \NN$ pairwise co-prime, monic non-constant polynomials $q_1(x), \dots, q_d(x) \in\FF_q[x]$ and denote their degrees, which are all positive, by $e_1,\ldots,e_d$. For $k \in \NN_0$, we construct the $k$-th column of the $j$-th generating matrix $C^{(j)}$, denoted by $\sigma^{(j)}_k = (\sigma^{(j)}_{t,k})_{t \geq 1}$, by using the representation of $x^{k}$ in terms of powers of $q_j(x)$, i.e., $x^k = \sum_{s\geq 0} b_s(x) q_j^s(x)$ with $b_s(x) \in \FF_q[x]$ satisfying $\deg{b_s(x)} < e_j$. This representation can be computed as follows: 
\begin{align*}
x^k &= a_0(x) q_j(x)+b_0(x),\quad \text{ where } a_0(x), b_0(x) \in \FF_q[x] \text{ such that } \deg{b_0(x)} < e_j,\\
a_0(x) &= a_1(x) q_j(x)+b_1(x),\quad \text{ where } a_1(x), b_1(x)\in \FF_q[x] \text{ such that } \deg{b_1(x)} < e_j,\\
& \,\,\, \vdots \,\,\, . 
\end{align*}
Note that there are just finitely many nonzero remainder polynomials $b_s(x)$.
Now we consider the representation of the remainder polynomial $b_s(x)$ in terms of powers of $x$, i.e. $b_s(x) = b_{s,0} + b_{s,1}x + \dots + b_{s,e_j - 1}x^{e_j - 1}$, and set 
$$(\sigma^{(j)}_{e_js+1,k}, \sigma^{(j)}_{e_js+2,k}, \dots, \sigma^{(j)}_{e_js+e_j,k}):= (b_{s,0},b_{s,1}, \dots, b_{s,e_j-1}).$$
}

The matrices $C^{(1)}, \dots, C^{(d)}$ are $NUT$ matrices and generate a $(0, \bse, d)$-sequence in base $q$. (Cf. \cite[proof of Theorem~1]{hofer}).

For both of these methods of constructing generating matrices of digital sequences, we analyze the left hand side of \eqref{equ:def_d} and obtain our first main result.
\begin{thm}
	\label{thm:Niederreiter}
	The digital $(0, \bm{e}, d)$-sequences with generating matrices $C^{(1)}, \dots, C^{(d)}$ obtained via the Niederreiter construction or the alternative column by column approach do not have Poissonian pair correlations.
\end{thm}

\subsection{Halton Sequences}
Other multi-dimensional point sequences which are of wide interest and which can be seen as the extension of the van der Corput sequence to higher dimensions are Halton sequences \cite{halton}.
\\

\textit{
	Let $d \in \N$, $b_1, \dots, b_d\geq 2$ be pairwise relatively prime integers and for $b\geq 2$ let $\phi_b: \N_0 \to \ui$ be the $b$-adic radical inverse function, defined as 
	$$\phi_b(n) := \frac{n_0}{b} + \frac{n_1}{b^2} + \dots$$
	where $n = n_0 + n_1b + \dots$ with $n_i \in \{0, \dots, b-1\}$ for $i \in \N_0$ is the unique base $b$ representation of $n$. The Halton sequence in bases $b_1, \dots, b_d$ is the sequence $\xn$ in $[0,1)^d$ whose elements are given by 
	$$\bm{x}_n = (\phi_{b_1}(n), \dots, \phi_{b_d}(n)).$$
}

Again, see e.g. \cite{DP10} for more details.
The question whether Halton sequences have Poissonian pair correlations was posed in \cite{HinrEtAl2018} and also stated as Problem~5 in \cite{LS2019}, although it was suggested that this is most likely not the case. It turns out that this conjecture indeed is true.
\begin{thm}
	\label{thm:Halton}
	The Halton sequence $\xn$ in pairwise relatively prime integer bases $b_1, \dots, b_d$, $d \in \NN$, does not have Poissonian pair correlations.
\end{thm}

Of course, it typically is expected that multi-dimensional versions of sequences have similar qualities as their one-dimensional analogues. However, it should be mentioned that an exceptional behaviour of Halton sequences has been observed for the instance of the $L_p$-discrepancy for $p<\infty$. Recently, Levin proved that higher-dimensional Halton sequences have optimal order of $L_p$-discrepancy \cite{levin}, even though the one-dimensional van der Corput sequence does not satisfy optimal $L_p$-discrepancy bounds (see e.g. \cite{pill}). 
\\

The rest of the paper is organized as follows. The key ingredient for the proofs of Theorem \ref{thm:Niederreiter} and Theorem \ref{thm:Halton} is a general tool stated as Proposition \ref{thm:general_result} at the beginning of the next section. It identifies regularity conditions of sequences which are sufficient for failing Poissonian pair correlations. Verifying these conditions for both, Niederreiter and Halton sequences, is not trivial and therefore takes the majority of Section \ref{sec:proofs}. Finally, in Section \ref{sec:discussion} we give an outlook to future research tasks and discuss a problem in algebraic number theory and Diophantine approximation that occurred during the investigation of Halton sequences.

\section{Proofs}\label{sec:proofs}
The property of Poissonian pair correlations can be seen as local quality criterion for a sequence $\xn$ to be uniformly distributed. Therefore, one might suggest that deterministically generated sequences which show a certain degree of regularity do not enjoy this feature. In fact, this is the statement of the following proposition, which serves as one of our key tools for the proofs of Theorems \ref{thm:Niederreiter} and \ref{thm:Halton}.

\begin{prop}
	\label{thm:general_result}
	Let $\xn$ be a sequence in $\ui^d$.
	If there exists a strictly increasing sequence of positive integers $(N_k)_{k \in \N}$ such that $(\bm{x}_{n})_{0 \leq n < N_k}$ fulfills 
	\begin{equation}
	\label{eq:estimate}
	\#\Big\{0 \leq n \neq l \leq N_k-1 : \|\bm{x}_n - \bm{x}_l\|_{\infty} \in \Big(\frac{a}{N_k^{1/d}},\frac{b}{N_k^{1/d}}\Big]\Big\} \geq c N_k
	\end{equation}
	for all $k$ larger than some index $k_0$ and where $a,b,c > 0$ are real constants which satisfy
	\begin{equation}
	\label{eq:assumption_a_b_c}
	c > (2b)^d - (2a)^d>0,
	\end{equation}
	then $\xn$ does not have Poissonian pair correlations.
\end{prop}

\begin{proof}
	To begin with, assume that $\xn$ has Poissonian pair correlations. We use this property for $s = b$ and obtain
	$$\frac{1}{N_k} \#\Big\{0 \leq n \neq l \leq N_k-1 : \|\bm{x}_n - \bm{x}_l\|_{\infty} \leq \frac{b}{N_k^{1/d}}\Big\} \rightarrow (2b)^d $$
	as $N_k \to \infty$. It holds that
	\begin{align*}
	\#\Big\{0 &\leq n \neq l \leq N_k-1 : \|\bm{x}_n - \bm{x}_l\|_{\infty} \leq \frac{b}{N_k^{1/d}}\Big\} \\
	&= \#\Big\{0 \leq n \neq l \leq N_k-1 : \|\bm{x}_n - \bm{x}_l\|_{\infty} \leq \frac{a}{N_k^{1/d}}\Big\} \\
	&\quad + \#\Big\{0 \leq n \neq l \leq N_k-1 : \|\bm{x}_n - \bm{x}_l\|_{\infty} \in \Big(\frac{a}{N_k^{1/d}},\frac{b}{N_k^{1/d}}\Big]\Big\} \\
	&=: A + B.
	\end{align*}
	Therefore, for any $\varepsilon_1 > 0$ there exists an index $k(\varepsilon_1)$ such that for all $k > k(\varepsilon_1)$ we have
	$$\frac{A}{N_k} + \frac{B}{N_k} \leq (2b)^d + \varepsilon_1.$$
	For sufficiently large $N_k$ we can use the assumptions \eqref{eq:estimate} and obtain
	$$\frac{A}{N_k} \leq (2b)^d + \varepsilon_1 - \frac{B}{N_k} \leq (2b)^d + \varepsilon_1 - c.$$
	
	Now consider $A/N_k$ which tends to $(2a)^d$ as $N_k \to \infty$ by the property of Poissonian pair correlations for $s = a$.
	Again this implies that for any $\varepsilon_2 > 0$ there is an index $k(\varepsilon_2)$ such that for all $k > k(\varepsilon_2)$ it holds that 
	$$\frac{A}{N_k} \geq (2a)^d - \varepsilon_2.$$
	By assumption \eqref{eq:assumption_a_b_c}, there exists $\kappa > 0$ such that 
	$$c = (2b)^d - (2a)^d + \kappa.$$
	However, if $\varepsilon_1$ and $\varepsilon_2$ are chosen such that $\varepsilon_1 + \varepsilon_2 < \kappa$ and provided that $N_k$ is sufficiently large we have 
	$$(2a)^d - \varepsilon_2\leq \frac{A}{N_k} \leq (2b)^d + \varepsilon_1 - c$$
	and
	$$c\leq (2b)^d-(2a)^d+ \varepsilon_1+\varepsilon_2<(2b)^d - (2a)^d + \kappa=c,$$
	which yields the desired contradiction to our assumption that  $\xn$ has Poissonian pair correlations.
\end{proof}

In the light of Proposition \ref{thm:general_result}, the key ingredient for proving that Niederreiter and Halton sequences do not have Poissonian pair correlations therefore is to find enough pairs of points, for which the distance between those points can be suitably well calculated and lies in a rather small interval.

\subsection{Application to digital sequences}
For the proof of Theorem \ref{thm:Niederreiter} we need some preliminary results.
\begin{lemma}(\cite[Prop.1]{FauTez})\label{lem:2}
	Let $C^{(1)},\ldots,C^{(d)}$ be the generating matrices of a digital $(t,\bse,d)$-sequence over $\FF_q$ and $S$ be a $NUT$ matrix in $\FF_q^{\NN_0\times\NN_0}$. Then $C^{(1)}S,\ldots,C^{(d)}S$ are generating matrices of a digital $(t,\bse,d)$-sequence over $\FF_q$. 
\end{lemma}
In the following the quantity $L_f$ denotes the maximal row length considering the first $f$ rows of all generating matrices. More precisely, taking the matrix consisting of the first $f$ rows of each of the generating matrices, $L_f-1$ is the index of the last non-zero column (or $\infty$, if none exists). 

\begin{lemma}\label{lem:S}
	Let $C^{(1)},\ldots,C^{(d)}$ be the generating matrices associated to the distinct monic non-constant pairwise co-prime polynomials $q_1(x),\ldots,q_d(x)$ with degrees $e_1,\ldots,e_d$ using one of the two constructions given in Section \ref{sec:digital_sequences}. We set $v:=\lcm(e_1,\ldots,e_d)$ and define the matrix $S\in\FF_q^{\NN_0\times\NN_0}$ as follows. 
	For $k\in\NN_0$, the $k$-th column $\sigma_k$ of $S$, is given by $\sigma_k = (b_0, b_1, \dots, b_{k-1}, 0, \dots)^\top$ where the $b_n$ are the coefficients of the following monic polynomial of degree $k$, 
	$$p_k(x)=x^{r_1}\prod_{i=1}^d q_i(x)^{(s_i+s_{i+1}+ \cdots +s_d)v/e_i} = \sum_{n\geq 0}b_nx^{n}.$$
	Here the $s_i$ and $r_i$ are defined as follows
	\begin{align*}
	k &=  dv s_d+r_d,& \, r_d \in\{0,\ldots,vd-1 \} \\
	r_d &= (d-1)vs_{d-1}+r_{d-1},& \, r_{d-1}\in\{0,\ldots,v(d-1)-1\} \\
	&\,\,\,\vdots & \vdots \,\,\, \\
	r_2&= v s_1+r_1,& \, r_1\in\{0,\ldots,v-1\}.
	\end{align*}
	Then the matrices $C^{(1)}S, \dots, C^{(d)}S$ generate a digital $(0,\bse,d)$-sequence and satisfy $L_f\leq d f$ provided that $v|f$. 
\end{lemma}

\begin{proof}
For the Niederreiter construction see \cite{HoPi12}, for the column by column construction note the linearity of the construction algorithm and the fact that $(1,x,x^2,\ldots,x^k)$ as well as $(1,p_1(x),p_2(x),\ldots,p_k(x))$ form a base of $\{p\in\FF_q[x]: \deg(p)\leq k\}$.
\end{proof}

\begin{refproof}[Proof of Theorem \ref{thm:Niederreiter}.]
We again may assume $d\geq 2$ as the one-dimensional case was treated already in \cite{LS2018}. Let $q_j(x)$, $j \in \{1, \dots, d\}$ be monic non-constant pairwise co-prime polynomials over $\FF_q$ of degrees $e_j \geq 1$ and let $C^{(j)}$ denote the corresponding generating matrices constructed via the Niederreiter or the alternative column by column approach. Moreover, let $m$ be a multiple of $v = \lcm(e_1, \dots, e_d)$ times $d$, i.e. $m = kvd$ with $k \in \NN$, let $M_m = q^m$ and $N_m = 2q^m$.

The sequence $(\bsx_n)\lc$ constructed via the digital method is a $(0,\bse, d)$-sequence, therefore any elementary interval with volume $q^{-m}$ of the form
$$I(vk, \dots, vk) = \prod_{j=1}^d\left[\frac{a_j}{q^{vk}},\frac{a_j+1}{q^{vk}}\right)$$ 
where $0 \leq a_j < q^{vk}$, contains exactly one point $\bsx_n$ with $n \in \{0, 1, \dots, M_m-1\}$ and one point $\bsx_l$ with $l \in \{M_m, M_m +1, \dots, N_m -1\}$. The idea of the proof then is to find infinitely many suitable values of $m$ such that the distances between the elements $\bsx_n$ and $\bsx_l$ are similar for many $n$ and $l$, respectively, in order to apply Proposition \ref{thm:general_result}.

To begin with, for arbitrary $n = n_0 + n_1q + \dots$ and $l = l_0 + l_1q + \dots$ we define
$$\Delta_{n,l} := (\phi^{-1}(l_0), \phi^{-1}(l_1), \dots, \phi^{-1}(l_{m}))^\top - (\phi^{-1}(n_0), \phi^{-1}(n_1), \dots, \phi^{-1}(n_{m}))^\top \in \FF_q^{m+1}.$$
Note that $n_m=0$ and $l_m=1$. 

The elements $\bsx_n$ and $\bsx_l$ lie in the same elementary interval $I(vk, \dots, vk)$ if
\begin{equation}
\label{eq:last_column_D}
D_{m\times (m+1)}\Delta_{n,l} = (0, \dots, 0)^\top,
\end{equation}
where $D_{m\times (m+1)} \in \FF_q^{m \times (m+1)}$ is the $(m \times (m+1))$-matrix whose rows consist of the rows of each upper left $(kv \times (m+1))$-submatrix of $C^{(j)}$, $j \in \{1, \dots, d\}$. Note that $D_{m\times (m+1)}$ has rank $m$.

Furthermore, let $S_{m+1}$ be the upper left $((m+1) \times (m+1))$-submatrix of $S$ defined in Lemma \ref{lem:S}. Since $S_{m+1}$ is regular we can rewrite
\begin{align*}
D_{m\times (m+1)} \Delta_{n,l} = D_{m\times (m+1)} S_{m+1} S^{-1}_{m+1}\Delta_{n,l}.
\end{align*}
Since $0 \leq n < M_m$, $M_m \leq l < N_m$ and both, $S_{m+1}$ and $S^{-1}_{m+1}$ are $NUT$ matrices with 1s in the diagonal, we have
$$S^{-1}_{m+1}\Delta_{n,l} = \begin{pmatrix}\bm{d}\\\phi^{-1}(1)\end{pmatrix}$$
with $\bm{d} \in \FF_q^m$.
Moreover, from Lemma \ref{lem:S} with $f = kv$ it follows that the last column of the product
$D_{m\times (m+1)} S_{m+1}$
consists of zeros exclusively. From \eqref{eq:last_column_D} it therefore follows that $\bm{d}$ is the zero vector in $\FF_q^m$, i.e.
$$S^{-1}_{m+1}\Delta_{n,l} = \begin{pmatrix} \bm 0 \\ \phi^{-1}(1) \end{pmatrix},$$
or equivalently after multiplying with $S_{m+1}$ from the left, 
$\Delta_{n,l}$ equals $\phi^{-1}(1)$ times the $(m+1)$-st column of $S_{m+1}$, which is determined by the representation of the polynomial
$$p_m(x)=\prod_{i=1}^dq_i(x)^{vk/e_i}$$
in terms of powers of $x$.

We now have to find specific choices of $m$ or $k$, respectively, such that we obtain special $\Delta_{n,l}$ in order to apply Proposition \ref{thm:general_result}. Therefore, let
$$\tau_i:=\min \{r \in \NN: q_i^{rv/e_i}(x) \equiv 1 \pmod{q_j^{v/e_j}(x)} \text{ for every } j \neq i\}.$$
Then we use the characteristic $p$ of the finite field $\FF_q$ and the fact that for all $k \in \NN$ and $1 \leq l < p^k$ we have $\binom{p^k}{l}\equiv 0 \pmod{p}$. So whenever $u$ is a power of the characteristic of $\FF_q$ we have for all $f(x), g(x) \in \FF_q[x]$ that 
$$(f(x)+g(x))^u=f^u(x)+g^u(x).$$
If we set $\theta:=\lcm(\tau_1 ,\ldots,\tau_d)$, we therefore have
$$q_i^{\theta v/e_i}(x)\equiv 1 \pmod{q_j^{v/e_j}(x)}$$
and 
\begin{equation}
\label{eq:prod_congruence}
\prod_{\substack{i = 1 \\ i \neq j}}^{d} q_i^{uv\theta /e_i}(x)\equiv 1 \pmod{q_j^{uv/e_j}(x)}.
\end{equation}
Then let $m = uv\theta d$ and consider $C^{(j)}_{(m+1) \times (m+1)}\Delta_{n,l}$.
\begin{enumerate}
	\item If the matrix $C^{(j)}$ is constructed via the Niederreiter approach, then the $k$-th entry of $C^{(j)}_{(m+1) \times (m+1)}\Delta_{n,l}$ is the coefficient of $x^{-1}$ in the Laurent series expansion of 
	$$\phi^{-1}(1) \frac{x^r}{q_j^s(x)} \prod_{i = 1}^{d}q_i^{uv\theta/e_i}(x),$$
	with $k = e_js - r$ and $r \in \{0,1, \dots, e_j-1\}$.
	For $s \leq uv\theta/e_j$ and any admissible value of $r$, the expression above is a polynomial and therefore the coefficient of  $x^{-1}$ in its Laurent series expansion is 0. For $uv\theta/e_j < s \leq uv\theta/e_j + uv/e_j$ we use \eqref{eq:prod_congruence} to get
	$$\frac{x^r}{q_j^{s}(x)}\prod_{\substack{i = 1}}^{d}q_i^{uv\theta/e_i}(x) = x^rb(x) + \frac{x^r}{q_j^{s-uv\theta/e_j}(x)},$$
	for some polynomial $b(x) \in \FF_q[x]$. Remember that $x^rq_j^{-(s-uv\theta/e_j)}(x)$ exactly determines row $e_j(s-uv\theta/e_j)-r$ of $C^{(j)}$ and that the coefficient of $x^{-1}$ of this expression is the entry in the first column. Since $C^{(j)}$ is a $NUT$ matrix with 1s in the diagonal we obtain
	$$C^{(j)}_{m+1 \times m+1}\Delta_{n,l} = (\underbrace{0, \dots, 0}_{uv\theta}, \phi^{-1}(1),\underbrace{0, \dots, 0}_{uv-1}, a, \dots)^\top,$$
	with $a \in \FF_q$.
	
	\item In case that the matrix $C^{(j)}$ is constructed via the column by column approach, the entries of $C^{(j)}_{(m+1) \times (m+1)}\Delta_{n,l}$ are given by the coefficients of the representation of $\phi^{-1}(1)p_m(x)$ in terms of powers of $q_j(x)$. Using \eqref{eq:prod_congruence} we get
	\begin{align*}
	p_m(x) &= \prod_{i = 1}^{d} q_i^{uv\theta/e_i}(x) \\
	&= q_j^{uv\theta/e_j}(x) \Big(\prod_{\substack{i = 1 \\ i \neq j}}^d q_i^{uv\theta/e_i}(x)\Big) \\
	&= q_j^{uv\theta/e_j}(x) \Big(1 + b(x) q_j^{uv/e_j}(x)\Big).
	\end{align*}
	for some polynomial $b(x) \in \FF_q[x]$.
	Thus, here we also have
	$$C^{(j)}_{(m+1) \times (m+1)}\Delta_{n,l} = (\underbrace{0, \dots, 0}_{uv\theta}, \phi^{-1}(1),\underbrace{0, \dots, 0}_{uv-1}, a, \dots)^\top,$$
	with $a \in \FF_q$.
\end{enumerate}
Assume now that the $(uv\theta +1)$-st entry $y^{(1)}_{uv\theta+1}$ of $C^{(1)} \cdot (\phi^{-1}(n_0), \phi^{-1}(n_1), \dots)^\top$ fulfills
$$|\phi(y^{(1)}_{uv\theta + 1} + \phi^{-1}(1)) - \phi(y^{(1)}_{uv\theta + 1})| = \max\{|\phi(\alpha + \phi^{-1}(1)) - \phi(\alpha)|: \text{ for } \alpha \in \FF_q\} =: w\geq 1,$$
which is the case for at least $M_m/q$ many values of $n \in \{0, 1, \dots, M_m-1\}$. Then, for such $n$ it holds that
$$\|\bsx_n - \bsx_l\|_\infty \in \Big[\frac{w}{q^{uv\theta +1}} - \frac{1}{q^{uv\theta + uv}}, \frac{w}{q^{uv\theta +1}} + \frac{1}{q^{uv\theta + uv}} \Big).$$
Finally, let $\varepsilon > 0$ and set 
$$a = 2^{1/d}\frac{w}{q} - \varepsilon, \qquad b = 2^{1/d}\frac{w}{q} + \varepsilon, \qquad c = \frac{1}{q}.$$
Thus, for $N_m = 2q^{m}$ with $m = uv\theta d$, where $u$ is a power of the characteristic of $\FF_q$, we have
$$\#\Big\{0 \leq n \neq l \leq N_m-1 : \|\bm{x}_n - \bm{x}_l\|_{\infty} \in \Big(\frac{a}{N_m^{1/d}},\frac{b}{N_m^{1/d}}\Big]\Big\} \geq c N_m$$
provided that $u$ is chosen large enough. 
However, since $\varepsilon$ can be chosen such that
$$ (2b)^d - (2a)^d = 2^d( (2^{1/d}\frac{w}{q} + \varepsilon)^d - (2^{1/d}\frac{w}{q} - \varepsilon)^d ) < \frac{1}{q} = c,$$
the assumptions of Proposition \ref{thm:general_result} are fulfilled and the considered sequences therefore do not have Poissonian pair correlations.
\end{refproof}

\subsection{Application to Halton sequences}
In order to be able to also apply Proposition \ref{thm:general_result} to Halton sequences, we again need a preliminary result, formulated as Lemma \ref{lem:nalpha_cluster_point} below. 

However, this lemma makes use of \textit{Minkowski's Theorem} (see \cite{minkowski}) that states that if  $C \subseteq \R^d$ is a convex set that is symmetric about the origin (i.e., $x \in C$ if and only if $-x \in C$) and with $\vol(C) > 2^d m$, then there are at least $m$ different points $\bsz_1, \dots, \bsz_m$ such that $\pm \bsz_1, \dots, \pm \bsz_m \in C \cap \Z^d \setminus \{0\}$.

\begin{lemma}
	\label{lem:nalpha_cluster_point}
	Let $d \in \N$ and $\alpha_1, \dots, \alpha_d$ be irrational. Then the sequence $(\{n \bm{\alpha}\})\lc$ in $\ui^d$ with $\{n \bm{\alpha}\} = (\{n\alpha_1\}, \dots, \{ n\alpha_d\})$ has an accumulation point in
	$$D:= \{(\delta_1, \dots, \delta_d): \delta_j \in \{0,1\}, j \in \{1, \dots, d\}\}.$$
\end{lemma}
\begin{proof}
For $N \in \N$ and arbitrary $\varepsilon_j > 0$, $j \in \{1, \dots, d\}$, define $C_N \in \R^{d+1}$,
	$$C_N:=\{(x_0,x_1, \dots, x_d) \in \R^{d+1}: |\alpha_j x_0 - x_j| \leq \varepsilon_j, j \in \{1, \dots, d\}, |x_0| \leq N\}.$$
	The set $C$ is convex and symmetric about the origin with 
	$$\text{vol}(C) = 2^{d+1} N \prod_{j = 1}^d\varepsilon_j.$$
	Therefore, if $N > m / (\prod_{j = 1}^d\varepsilon_j)$, we have $\text{vol}(C) > 2^{d+1} m$ and, by Minkowski's Theorem, there exist $m$ different elements $\bsz_i = (z_i^{(0)}, z_i^{(1)}, \dots, z_i^{(d)})$, $i \in \{1, \dots, m\}$ with $\bsz_i \in C \cap \Z^{d+1} \setminus \{0\}$ and $z_i^{(0)} \geq 0$. Moreover, for those elements it holds that $|\alpha_j z_i^{(0)} - z_i^{(j)}|\leq \varepsilon_j$, thus $\{\alpha_j z_i^{(0)}\} \in (0,\varepsilon_j] \cup [1-\varepsilon_j,1)$ for all $j \in \{1, \dots, d\}$. Note that, if $\varepsilon_j$ are chosen small enough the integers $z_i^{(0)}$ will be distinct. 
\end{proof}

\begin{refproof}[Proof of Theorem \ref{thm:Halton}.]
For $d=1$ we have to consider the van der Corput sequence for which it is well-known that it does not have Poissonian pair correlations. Hence we assume $d\geq 2$ in the following. 
Let $b_1, \dots, b_d$ be pairwise relatively prime integers and let $\xn$ denote the Halton sequence in bases $b_1, \dots, b_d$. Without loss of generality we assume $b_1 < b_j$ for all $j \in \{2, \dots, d\}$.

Let $u \in \N\setminus\{1\}$ and define
\begin{align*}
P_1 &:= \prod_{j = 2}^{d}b_j^2, \qquad 
P_i := b_1^u\left( \prod_{\substack{j = 2 \\ j \neq i}}^{d}b_j^2 \right), \\
\tau_1 &:= \min\{1 \leq l \leq P_1: b_1^{ul} \equiv 1 \pmod{P_1}\}, \\
\tau_i &:= \min\{1 \leq l \leq P_i: b_i^{2l} \equiv 1 \pmod{P_i}\}
\end{align*}
for all $i \in \{2, \dots, d\}$. Such $\tau_1, \tau_i$ exist as $\gcd(P_1,b_1)=\gcd(P_i,b_i)=1$ and $d\geq 2$. 

Similar as in the proof of Theorem \ref{thm:Niederreiter} we define for $\bm{k} = (k_1, \dots, k_d)\in\N_0^d$ numbers $N_{\bm{k}}\in\N$ and corresponding subintervals
$$I := I(u\tau_1k_1, 2\tau_2k_2, \dots, 2\tau_dk_d) =\left[\frac{a_1}{b_1^{u\tau_1k_1}},\frac{a_1+1}{b_1^{u\tau_1k_1}}\right)\times\prod_{j=2}^d\left[\frac{a_j}{b_j^{2\tau_jk_j}},\frac{a_j+1}{b_j^{2\tau_jk_j}}\right),$$
where $0 \leq a_1 < b_1^{u\tau_1k_1}$ and $0\leq a_j<b_j^{2\tau_jk_j}$ for $j \in \{2, \dots, d\}$, and study the distances between the points $\bm{x}_n$ that lie in the same subinterval $I$. 

Now let
\begin{align*}
	M &= M(\bm{k}) := b_1^{u\tau_1k_1} \left(\prod_{j = 2}^{d} b_j^{2\tau_j k_j}\right), \\
	L &= L(\bm{k}) := b_1^{u\tau_1k_1+1} \left(\prod_{j = 2}^{d} b_j^{2\tau_j k_j+1}\right).
\end{align*}
By a special regularity of the sequence, which is an easy consequence of the Chinese Remainder Theorem, we have that exactly $\prod_{j = 1}^{d}b_j$ points of the first $L$ points and exactly one point of the subsequent $M$ points of the sequence lie in $I$. Moreover, $\bm{x}_{n+M}\in I$ if and only if $\bm{x}_{n}\in I$. 

We set $N_{\bm{k}}:=L+M$ and study $\|\bm{x}_n - \bm{x}_{n + M}\|_\infty$ for $0\leq n<L$. 

By $(n)_{b_j}$ we denote the digit representation of $n$ in base $b_j$, i.e. for $n = n_0 + n_1 b_j + n_2 b_j^2 + \dots$ we have $(n)_{b_j} = (n_0, n_1, n_2, \dots)$. Note that obviously $b_1^{u\tau_1k_1} |M$ and  $b_j^{2\tau_jk_j}|M$. By the choice of $\tau_1$ and $\tau_j$ we have
$$b_1^{u\tau_1} \equiv 1 \pmod{b_j^2}$$
and also for $i \neq j$, 
$$\qquad b_i^{2\tau_i}\equiv 1 \pmod{b_j^{2}} \qquad \text{ and } \qquad b_i^{2\tau_i}\equiv 1 \pmod{b_1^{u}}.$$
Therefore, 
$$\prod_{i = 2}^d b_i^{2\tau_ik_i} \equiv 1 \pmod{b_1^u} \qquad \text{ and } \qquad b_1^{u\tau_1k_1}\prod_{i=2,i\neq j}^db_i^{2\tau_ik_i}\equiv 1 \pmod{b_j^{2}}.$$
Hence, 
\begin{equation}
	\label{eq:Mdigits}
	\begin{split}
	(M)_{b_1} &= (\underbrace{0,\dots, 0}_{u\tau_1k_1}, 1, \underbrace{0, \dots, 0}_{u -1}, m_{u\tau_1 k_1 + u}, \dots), \\
	(M)_{b_j} &= (\underbrace{0,\dots, 0}_{2\tau_j k_j}, 1, 0, m_{2\tau_j k_j + 2}, \dots).
	\end{split}
\end{equation}	

Now consider $\|\bm{x}_n - \bm{x}_{n + M}\|_\infty = \sup_{j \in \{1, \dots, d\}} \|x^{(j)}_n - x^{(j)}_{n + M}\|$. If for $(n)_{b_1}$ it holds that if $n_{u\tau_1 k_1} \neq b_1 -1$ then by \eqref{eq:Mdigits} the first $u \tau_1 k_1 + u$ entries except of $(u \tau_1 k_1+1)$-st entry of $(n+M)_{b_1}$ and $(n)_{b_1}$ coincide. As $\sum_{i=m+1}^\infty \frac{b_1-1}{b_1^{i}}=\frac{1}{b_1^m}$
we have in the case where $n_{u \tau_1 k_1} \neq b_1-1$,
\begin{align*}
	\|x^{(1)}_n - x^{(1)}_{n + M}\| 
	&\in \Big(\frac{1}{b_1^{u \tau_1k_1 +1}} - \frac{1}{b_1^{u \tau_1 k_1 + u}}, \frac{1}{b_1^{u \tau_1k_1 +1}} + \frac{1}{b_1^{u \tau_1 k_1 + u}}\Big).
\end{align*}
Similarly, for the other coordinates $j \in \{2, \dots, d\}$ we obtain in the case where in $(n)_{b_j}$ we have $n_{2 \tau_j k_j} \neq b_j-1$,
\begin{align*}
\|x^{(j)}_n - x^{(j)}_{n + M}\|
&\in \Big(\frac{1}{b_j^{2 \tau_jk_j +1}} - \frac{1}{b_j^{2 \tau_j k_j + 2}}, \frac{1}{b_j^{2 \tau_jk_j +1}} + \frac{1}{b_j^{2 \tau_j k_j + 2}}\Big).
\end{align*}
\ \\

Next, we want to find constants $\xi_j\geq 1$, $j \in \{2, \dots, d\}$, such that
\begin{equation}
\label{eq:estimate_fraction} 
\xi_j \leq \frac{b_j^{2\tau_jk_j +1}}{b_1^{u\tau_1k_1 +1}} \leq \xi_j f(u)
\end{equation}
with 
$$f(u) := \left(\frac{1+b_1^{1-u}}{1-b_1^{1-u}}\right)^{\frac{d}{d-1}}$$
is simultaneously fulfilled for infinitely many $\bm{k} = (k_1, k_2, \dots, k_d) \in \N_0^d$ and thus also for infinitely many $N_{\bm{k}} = M + L$. Therefore, we define $\beta_1 := b_1^{u \tau_1}$ and $\beta_j := b_j^{2\tau_j}$ for $j \in \{2, \dots, d\}$. The inequalities in \eqref{eq:estimate_fraction} are then equivalent to
\begin{equation}
\label{eq:xi}
\log_{\beta_j}\Big(\xi_j\frac{b_1}{b_j}\Big) + k_1 \log_{\beta_j}(\beta_1) \leq k_j \leq \log_{\beta_j}\Big(\xi_j f(u)\frac{b_1}{b_j}\Big) + k_1 \log_{\beta_j}(\beta_1).
\end{equation}
Moreover, we consider the sequence $(\{n \bm{\alpha}\})\lc \in \ui^{d-1}$ with $\{n \bm{\alpha} \} = (\{n \alpha_2\}, \dots, \{n \alpha_d\})$ and $\alpha_j = \log_{\beta_j}(\beta_1) \in \R \setminus \mathbb{Q}$.
Let now $(\delta_2, \dots, \delta_d) \in \{0,1\}^{d-1}$ denote an accumulation point of this sequence which exists by Lemma \ref{lem:nalpha_cluster_point}.

We want to distinguish two cases:
If $\delta_j = 0$ we set 
$$\xi_j := \frac{b_j}{b_1} \frac{1}{f(u)} \qquad \text{ and } \qquad k_j := \lfloor k_1 \log_{\beta_j}(\beta_1)\rfloor.$$
Note that $\xi_j > 1$ if $u$ is large enough. The inequalities \eqref{eq:xi} are then equivalent to 
$$\{k_1 \log_{\beta_j}(\beta_1)\} - \log_{\beta_j}(f(u)) \leq 0 \leq \{k_1 \log_{\beta_j}(\beta_1)\},$$
which is fulfilled if 
\begin{equation}
\label{eq:delta_0}
\{k_1 \log_{\beta_j}(\beta_1)\} \in \big[0, \log_{\beta_j}(f(u))\big].
\end{equation}

If $\delta_j = 1$ we set 
$$\xi_j := \frac{b_j}{b_1} \qquad \text{ and } \qquad k_j := \lfloor k_1 \log_{\beta_j}(\beta_1)\rfloor +1.$$
Again, $\xi_j > 1$ and \eqref{eq:xi} is equivalent to 
$$\{k_1 \log_{\beta_j}(\beta_1)\} \leq 1 \leq \{k_1 \log_{\beta_j}(\beta_1)\} + \log_{\beta_j}(f(u)),$$
which is fulfilled if 
\begin{equation}
\label{eq:delta_1}
\{k_1 \log_{\beta_j}(\beta_1)\} \in \big[1-\log_{\beta_j}(f(u)),1\big].
\end{equation}

By the fact that $(\delta_2, \dots, \delta_d)$ is an accumulation point of $(\{n \bm{\alpha}\})\lc$ with $\bm{\alpha} = (\log_{\beta_2}(\beta_1), \dots, \log_{\beta_d}(\beta_1))$, conditions \eqref{eq:delta_0} and \eqref{eq:delta_1}, respectively, are fulfilled simultaneously for each $j = 2, \dots, d$ for infinitely many $k_1$. Since $k_j \geq 0$ for all $j \in \{2, \dots, d\}$,  we know that there are also infinitely many $N_{\bm{k}}$ such that \eqref{eq:estimate_fraction} is fulfilled.

We can now use this important estimate to deduce that 
\begin{align*}
\frac{1}{b_1^{u \tau_1 k_1 + 1}} - \frac{1}{b_1^{u \tau_1 k_1 + u}} > \frac{1}{b_j^{2 \tau_j k_j + 1}} + \frac{1}{b_j^{2 \tau_j k_j + 2}}
\end{align*}
for all $j \in \{2, \dots, d\}$ and $u$ large enough.
This can be seen since 
\begin{align*}
\frac{b_j^{2 \tau_j k_j + 1}}{b_1^{u \tau_1 k_1 + 1}}\Big(1 - \frac{1}{b_1^{u -1}}\Big) &\geq \xi_j \Big(1 - \frac{1}{b_1^{u -1}}\Big) \\
&> \frac{b_j}{b_1} \frac{1}{f(u)}\Big(1 - \frac{1}{b_1^{u-1}}\Big) \\
&> \Big(1 + \frac{1}{b_j}\Big),
\end{align*}
where in the last step we used that $(1-1/b_1^{u-1})/f(u) \to 1$ as $u \to \infty$, and $b_j > b_1$.
Therefore, if in $(n)_{b_1}$ we have that $n_{u \tau k_1} \neq b_1 -1$ and in $(n)_{b_j}$, $j \in \{2, \dots, d\}$ we have that $n_{2\tau_j k_j} \neq b_j-1$,
$$\|\bm{x}_n - \bm{x}_{n+M}\|_\infty = \|x_n^{(1)} - x_{n+M}^{(1)}\| \in \Big(\frac{1-b_1^{1-u}}{b_1^{u \tau_1 k_1 +1}},\frac{1+b_1^{1-u}}{b_1^{u \tau_1 k_1 +1}}\Big).$$

As next step, we want to establish suitable bounds for $\|\bm{x}_n - \bm{x}_{n+M}\|_\infty $ in order to be able to apply Proposition~\ref{thm:general_result}, i.e. we want to show that there exist $a$ and $b$ such that
\begin{equation*}
\frac{a}{(L+M)^{1/d}} \leq \frac{1-b_1^{1-u}}{b_1^{u \tau_1 k_1 +1}} < \frac{1+b_1^{1-u}}{b_1^{u \tau_1 k_1 +1}} \leq \frac{b}{(L+M)^{1/d}},
\end{equation*}
which is equivalent to
\begin{equation}
\label{eq:a_b_Halton}
a^d \leq \frac{(1-b_1^{1-u})^d (L+M)}{(b_1^{u \tau_1 k_1 +1})^d} < \frac{(1+b_1^{1-u})^d (L+M)}{(b_1^{u \tau_1 k_1 +1})^d} \leq b^d.
\end{equation}
Note that 
$$L+M = b_1^{u \tau_1 k_1 +1}\left(\prod_{j = 2}^{d} b_j^{2\tau_j k_j+1}\right)\underbrace{\left(1 + \prod_{j = 1}^{d} b_j^{-1}\right)}_{:=\gamma^d}=L\gamma^d.$$
Using the estimate \eqref{eq:estimate_fraction} we find that \eqref{eq:a_b_Halton} is fulfilled if we choose 
\begin{align*}
a^d &:= (1 - b_1^{1-u})^d \left(\prod_{j = 2}^d\xi_j\right)\gamma^d, \\
b^d &:= \frac{(1 + b_1^{1-u})^{2d}}{(1 - b_1^{1-u})^d} \left(\prod_{j = 2}^d\xi_j\right)\gamma^d.
\end{align*}

Hence we have shown that
$$\|\bm{x}_n - \bm{x}_{n+M}\|_\infty \in \Big(\frac{a}{(L+M)^{1/d}}, \frac{b}{(L+M)^{1/d}}\Big)$$
for $n \in \{0, \dots, L -1\}$ whenever in $(n)_{b_1}$ we have $n_{u \tau k_1} \neq b_1 -1$ and in $(n)_{b_j}$, $j \in \{2, \dots, d\}$ we have $n_{2\tau_j k_j} \neq b_j-1$. Since this is the case for exactly $\big(\prod_{j = 1}^{d} \frac{b_j-1}{b_j}\big)L$ values of $n$ and each pair has to be counted twice in the pair correlation function, we obtain
\begin{align*}
\#\Big\{0 \leq n &\neq l \leq L+M-1: \|\bm{x}_n - \bm{x}_l\|_\infty \in \Big(\frac{a}{(L+M)^{1/d}}, \frac{b}{(L+M)^{1/d}}\Big]\Big\} \\
&\geq \#\Big\{0 \leq n \leq L-1: \|\bm{x}_n - \bm{x}_{n+M}\|_\infty \in \Big(\frac{a}{(L+M)^{1/d}}, \frac{b}{(L+M)^{1/d}}\Big]\Big\} \\
&\geq 2 \Big(\prod_{j = 1}^{d}\frac{b_j-1}{b_j}\Big)\frac{L}{L+M} (L+M) \\
&= 2 \Big(\prod_{j = 1}^{d}\frac{b_j-1}{b_j}\Big) \frac{1}{\gamma^d}(L +M)=:c\cdot (L+M).
\end{align*}

In order to apply Proposition \ref{thm:general_result} it therefore has to hold that
\begin{align}
\label{eq:condition_on_c}
(2b)^d - (2a)^d < 2 \Big(\prod_{j = 1}^{d}\frac{b_j-1}{b_j}\Big) \frac{1}{\gamma^d}=c.
\end{align}
Using the definition of $a$ and $b$ and the fact that $\xi_j \leq b_j /b_1$ we obtain
\begin{align*}
(2b)^d - (2a)^d &= 2^d \left(\prod_{j = 2}^d\xi_j\right) \gamma^d \Big(\frac{(1+b_1^{1-u})^{2d}}{(1-b_1^{1-u})^d} - (1- b_1^{1-u})^d\Big) \\
&\leq 2^d \left(\prod_{j = 2}^d\frac{b_j}{b_1}\right) \gamma^d \underbrace{\Big(\frac{(1+b_1^{1-u})^{2d} - (1- b_1^{1-u})^{2d}}{(1-b_1^{1-u})^d}\Big)}_{\text{tends to 0 for } u \to \infty}.
\end{align*}
Thus, if $u$ is chosen large enough, condition \eqref{eq:condition_on_c} is true and the Halton sequence in bases $b_1, \dots, b_d$ does not have Poissonian pair correlations.
\end{refproof}

\section{Discussion and Further Research}\label{sec:discussion}
Although Theorem \ref{thm:Niederreiter} of the present paper deals with a very prominent class of $(t,s)$-sequences, as a consequence of this result, of course a further research question is, whether other classes of $(t,s)$-sequences, as for example generalized Niederreiter sequences \cite{Tez93}, Niederreiter--Xing sequences \cite{NieXin} or even (digital) $(t,s)$-sequences in general, have the property of Poissonian pair correlation or not. \\

Furthermore, we would like to note an interesting relation of our method of proof to a conjecture in algebraic and transcendental number theory. During the search for a proof of Theorem \ref{thm:Halton} we faced the problem to simultaneously satisfy the inequalities
\eqref{eq:xi} with $\xi_j\geq 1$ for infinitely many $(k_1,k_2,\ldots,k_d)\in\NN_0^d$ to make sure that 
\begin{align*}
\frac{1}{b_1^{u \tau_1 k_1 + 1}} - \frac{1}{b_1^{u \tau_1 k_1 + u}} > \frac{1}{b_j^{2 \tau_j k_j + 1}} + \frac{1}{b_j^{2 \tau_j k_j + 2}}
\end{align*}
for all $j \in \{2, \dots, d\}$ and $u$ large enough. 

Note that if $1,\log_{\beta_2}\beta_1,\ldots,\log_{\beta_d}\beta_1$ were linearly independent over $\Q$ then the sequence $(\{n(\log_{\beta_2}\beta_1,\ldots,\log_{\beta_d}\beta_1)\})_{n\geq 0}\in[0,1)^{d-1}$ would be uniformly distributed in $[0,1)^{d-1}$. Such a statement would considerably shorten the proof of Theorem \ref{thm:Halton}. Unfortunately, it is not known whether for example the three numbers $1/\log 2,1/\log 3,1/\log 5$ are linearly independent over $\Q$ or not. The algebraic independence of the  logarithm of the prime numbers would be one consequence of the so-called Schanuel's conjecture in algebraic and transcendental number theory. We refer the interested reader to \cite{waldschmidt} for more details on this conjecture and its related problems.  

\section*{Acknowledgments}
We would like to thank Michel Waldschmidt for his e-mail correspondence concerning Schanuel's conjecture and the reciprocal of the logarithm of prime numbers.

\textbf{Author’s Addresses:} \\ 
Roswitha Hofer and Lisa Kaltenböck, Institut für Finanzmathematik und Angewandte Zahlentheorie, Johannes Kepler Universität Linz, Altenbergerstraße 69, A-4040 Linz, Austria. \\
Email: \url{roswitha.hofer@jku.at}, \url{lisa.kaltenboeck@jku.at}

\end{document}